\documentclass[12pt]{amsart}
\usepackage{geometry}                % See geometry.pdf to learn the layout options. There are lots.
\usepackage{graphicx}
\usepackage{amssymb}
\usepackage{multirow}
\usepackage{longtable}
\usepackage[colorlinks=true,
            linkcolor={magenta},citecolor={magenta},
            ]{hyperref}
\usepackage{tikz}
\usepackage{cite}

\newtheorem{theorem}{Theorem}[section]
\theoremstyle{plain}
\newtheorem{corollary}[theorem]{Corollary}

\newtheorem{example}[theorem]{Example}
\newtheorem{lemma}[theorem]{Lemma}
\newtheorem{prop}[theorem]{Proposition}
\newtheorem{remark}[theorem]{Remark}
\numberwithin{equation}{section}

\def\nn{\nonumber}
\def\nse{\mathrm{nse}}

\def\bexp{\mathrm{e}_\beta}
\def\pexp{\mathrm{e}}
\def\blog{\log_\beta}

\def\W{W}

\def\tn#1{\frac{t^#1}{#1!}}

\def\sone#1#2{\begin{bmatrix}#1 \\ #2 \end{bmatrix}}
\def\stwo#1#2{\left\{\begin{matrix}#1 \\ #2 \end{matrix}\right\}}

\def\pa{\alpha}
\def\pb{\beta}
\def\pc{\lambda}
\def\pd{\mu}
\def\Uxp#1#2{U_{#1}(x|#2)}
\def\pv{\mathbf{v}}
\def\Ffun{F}
\def\part#1#2{\frac{\partial#1}{\partial #2}}
\def\frdc{\left(\frac{\pd}{\pc}\right)}
\def\Uc#1#2#3{\mathcal{S}_{#1,#2}(#3)}
\def\nsb{\mathrm{nsb}}
\def\nse{\mathrm{nse}}
\def\rlb{\mathrm{rlb}}
\def\rle{\mathrm{rle}}
\def\Sfun{\mathbb{S}}
\def\LLP{\mathrm{LLP}}

\title{Unified generating function for set partitions}

\begin{document}
\author[O.~Herscovici]{Orli Herscovici}
\address{O.~Herscovici\\School of Mathematics, 
Georgia Institute of Technology,
686 Cherry St NW, Atlanta, GA 30332}
\email{orli.herscovici@gmail.com}
\maketitle

\begin{abstract}
In this work we define a unified generating functions for 9 different kinds of set partitions including cyclically ordered set partitions. Such generating function depends on 4 parameters. We consider property of this function and provide combinatorial explanation for polynomials generated by this function. Two new combinatorial statistics are defined and the explicit formulae given for coefficients of parametrized polynomials defined by the generating function.
\end{abstract}

\noindent{\sc Keywords:} set partitions statistics; Stirling numbers; degenerate exponential function;  permutations; cyclically ordered partitions; nonextensive statistical mechanics;

\noindent{\sc 2010 MSC:}
05A10; 05A15; 05A18; 05A30; 11B73; 

\section{Introduction}
A set partition problem has a very long history (for nice historical survey see for example the book of Mansour \cite{Mansour2013}). It is known that the number of partitions of a set $[n]$ into $k$ non-empty disjoint blocks is given by the Stirling numbers of the second kind $S(n,k)$. Many different statistics were proposed to study the nature of set partitions \cite{Dahlberg2016,  Mansour2013, Yano2007}.
Applications $q$-calculus \cite{Kac2002} to the theory of set partitions led to generalization of the Stirling numbers to the $q$-Stirling numbers $S_q(n,k)$ (see for example \cite{Simsek2012} and reference therein) and to  
new statistics on set partitions \cite{Chen2019, Kasraoui2009, Mendez2008, Park1994, Steingrimsson2020, White1994}. There exist other generalizations of Stirling numbers, but not all of them have a combinatorial interpretation (see for example \cite{Komatsu2018, Lang2000, Maltenfort2020, Mansour2012}). 
Another generalization was done in work with Mansour \cite{Herscovici2017}  and considered further in \cite{Herscovici2019}. It provided a single generating function for four kinds of set partitions (namely, sets of sets, sets of lists, lists of sets, and lists of lists \cite{Herscovici2017, Herscovici2019, Motzkin1971}). 

That generalization was based on replacement the exponential function by its deformed analogue $\text{exp}_q(t)=(1+(1-q)t)^\frac{1}{1-q}$ proposed by Tsallis in 1988 \cite{Tsallis1988, Tsallis1994} based on the experimental data.  This function is intimately close to the degenerate exponential function $(1+\lambda x)^\frac{1}{\lambda}$ of Carlitz proposed in 1956 \cite{Carlitz1956}.  The applications of the deformed exponential function of Tsallis found their applications in nonextensive statistical mechanics \cite{Borges1998, Lobao2009, Tsallis2009}, while degenerate exponential function of Carlitz mostly appears in deformations of generating functions arising in number theory \cite{Carlitz1979, Cenkci2007, Qi2018, Wu2014, Young2008}. 
Another difference between two functions is that deformed exponential function of Tsallis led to another kind of $q$-calculus  \cite{Borges1998, Lobao2009}.

In personal communication with Alan Sokal, he asked the following question: what happens if blocks or elements are cyclically ordered? Is there a single generating function describing not four, but nine kinds of set partitions? This work presents an affirmative answer on the last question. More interesting is that the obtained function is very similar to the function for four kinds of set partitions \cite{Herscovici2017} and the coefficients of this function are the same as for function describing the four previous set partitions \cite{Herscovici2019}.
The combinatorial part demands introducing of two complementary statistics to those defined in \cite{Herscovici2017}.

This paper is organized as following. In Section \ref{background} we present some basic notations. In Section \ref{ugf-sec} we introduce the unified generating function and study its properties. Section \ref{Section-Comb} considers the combinatorial interpretation of the unified generating function. Finally, Section \ref{sec-app} shows the applications of our function.

\section{Background}\label{background}

\subsection{Degenerate exponential and logarithmic functions}
Due to combinatorial nature of the objects considered here, we will use the following notations in this work. A degenerate exponential function \cite{Carlitz1956} will be denoted as
\begin{align}
\bexp(t)=(1-\beta t)^{-\frac{1}{\beta}}=\sum_{n=0}^\infty\W_n(\beta)\frac{t^n}{n!}, \label{GF:bexp}
\end{align}
where 
\begin{align}
\W_n(\beta)=\left\{\begin{matrix}
\prod_{m=1}^{n-1}(m\beta+1), & n\geq 2,\\
1, & n\in\{0,1\},\end{matrix}\right.\label{Wn}
\end{align}
(see \cite{Borges1998, Carlitz1979}).
\begin{remark} \label{p1911-rmk1}
Obviously, the coefficients $\W_n(\beta)$ are polynomials of degree $n-1$ in a variable $\beta$.
\end{remark}
It is easy to see that
\begin{align}
\frac{d}{dt}\bexp(t)=(1-\beta t)^{-\frac{1}{\beta}
-1}=\frac{\bexp(t)}{1-\beta t}.\label{bexp-deriv}
\end{align}
A degenerate logarithmic function (see also \cite{Tsallis1988}) corresponding to \eqref{GF:bexp} is therefore
\begin{align}
\blog(t)=\frac{t^{-\beta}-1}{-\beta}.\label{GF:blog}
\end{align}
However, in all generating functions related to set partitions, appearing in the Table \ref{longtable} at the end of this work, we actually need not a $\blog(t)$, but rather a $-\blog(1-t)$. So let's consider $-\blog(1-t)$. 

From \eqref{GF:blog} we obtain
\begin{align}
-\blog(1-t)&=-\frac{(1-t)^{-\beta}-1}{-\beta}=
\frac{1}{\beta}\, \pexp_{\frac{1}{\beta}}(\beta t)-\frac{1}{\beta}.\label{GF:logexp}
\end{align}
Let us emphasize that this function approaches $-\log(1-t)$ as $\beta\rightarrow0$, while function $\pexp_\beta(t)$ from \eqref{GF:bexp} approaches $\pexp^t$ as $\beta\rightarrow0$. Here, $\log t$ stands for the natural logarithm. Therefore let us consider a function 
\begin{align}
\pexp_\frac{\beta}{\alpha}(\alpha t)=\left(1-\frac{\beta}{\alpha}\cdot\alpha t\right)^{-\frac{\alpha}{\beta}}=\sum_{n=0}^\infty\W_n\left(\frac{\beta}{\alpha}\right)\frac{\alpha^nt^n}{n!},\label{GF:eab}
\end{align}
such that 
\begin{align}
\alpha\cdot\beta=0\implies ((\alpha=0)\text{ and }(\beta\neq 0))\text{ OR }((\beta=0)\text{ and }(\alpha\neq0)),\label{eab:cond}
\end{align}
i.e. parameters $\pa,\pb$ can not be equal  zero simultaneously. 
Considering the degenerate exponential function of kind \eqref{GF:eab}, the equality of any of its parameters, i.e. $\pa$ (or $\pb$) to the zero  should be addressed as a limit of $\alpha\rightarrow 0$ ($\pb\rightarrow 0$, respectively). It follows from \eqref{GF:eab} that $\pexp_\frac{\beta}{\alpha}(\alpha t)\Big|_{t=0}=1$ while from  \eqref{GF:logexp} we get $-\blog(1-t)\Big|_{t=0}=0$. 
We are ready to define a unified generating function.

\section{Unified generating function} \label{ugf-sec}

Based on similarities and differences in behaviours of functions $-\blog(1-t)$ and $\bexp(t)$, we construct the following function

\begin{align}
\Ffun(x,t;\pa,\pb,\pc,\pd)
=\frac{1}{\pa}\pexp_\frac{\pb}{\pa}\left[\pa x\left(\frac{1}{\pc}\pexp_\frac{\pd}{\pc}(\pc t)-\frac{1}{\pc}\right)\right]-\frac{1}{\pa}+\frac{\delta_{\pa,1}}{\pa},\label{Ffun}
\end{align}
where $\delta_{n,m}$ is the Kronecker's symbol and each pair of parameters ($\pa$, $\pb$) and ($\pc$, 
$\pd$) satisfy the condition \eqref{eab:cond}. We'll show that the function \eqref{Ffun} generates polynomials $\Uxp{n}{\pa,\pb,\pc,\pd}$ in variables $x$, $\pa$, $\pb$, $\pc$, and $\pd$. For simplicity of notations, let us denote by $\pv$ a vector of variables $\pa$, $\pb$, $\pc$, and $\pd$, that is $\pv:=(\pa,\pb,\pc,\pd)$ and, respectively $\Uxp{n}{\pv}:=\Uxp{n}{{\pa,\pb,\pc,\pd}}$. Thus we have
\begin{align}
\Ffun(x,t;\pv)=\sum_{n=0}^\infty\Uxp{n}{\pv}\tn{n}.\label{GF:UGF}
\end{align}

\begin{lemma} \label{lemma1}
We have $\Uxp{0}{{\pa,\pb,\pc,\pd}}=\frac{\delta_{\pa,1}}{\pa}$.
\end{lemma}

\begin{proof}
From \eqref{GF:UGF} we have $\Uxp{0}{\pv}=
\Ffun(x,0;\pv)$, so let us substitute $t=0$ into \eqref{Ffun}.
\begin{align}
\Uxp{0}{\pv}&=\frac{1}{\pa}\pexp_\frac{\pb}{\pa}\left[\pa x\left(\frac{1}{\pc}\pexp_\frac{\pd}{\pc}( 0)-\frac{1}{\pc}\right)\right]-\frac{1}{\pa}+\frac{\delta_{\pa,1}}{\pa}\nn\\
&=\frac{1}{\pa}\pexp_\frac{\pb}{\pa}\left[\pa x\left(\frac{1}{\pc}\cdot 1-\frac{1}{\pc}\right)\right]-\frac{1}{\pa}+\frac{\delta_{\pa,1}}{\pa}\nn\\
&=\frac{1}{\pa}\pexp_\frac{\pb}{\pa}(0)-\frac{1}{\pa}+\frac{\delta_{\pa,1}}{\pa}
=\frac{\delta_{\pa,1}}{\pa},
\end{align}
and the proof is complete.
\end{proof}
\begin{theorem} \label{thm2}
For all nonnegative integers $n$, $\Uxp{n}{\pv}$ satisfy the recurrence relation
\begin{align}
\Uxp{{n+1}}{\pv}=x\sum_{m=1}^n\left[\pa\binom{n}{m}+\pb\binom{n}{m-1}\right]\pc^{n-m}\W_{n-m+1}\left(\frac{\pd}{\pc}\right)\Uxp{m}{\pv}
+x\pc^n\W_{n+1}\left(\frac{\pd}{\pc}\right).\label{UGFrec1}
\end{align}
\end{theorem}
\begin{proof}
Let us differentiate \eqref{GF:UGF} w.r.t. $t$.
\begin{align}
\sum_{n=0}^\infty\Uxp{{n+1}}{\pv}\tn{n}&=\part{}{t}\left[\frac{1}{\pa}\pexp_\frac{\pb}{\pa}\left[\pa x\left(\frac{1}{\pc}\pexp_\frac{\pd}{\pc}(\pc t)-\frac{1}{\pc}\right)\right]-\frac{1}{\pa}+\frac{\delta_{\pa,1}}{\pa}\right]\nn\\
&=\frac{1}{\pa}\part{}{t}\pexp_\frac{\pb}{\pa}\left[\pa x\left(\frac{1}{\pc}\pexp_\frac{\pd}{\pc}(\pc t)-\frac{1}{\pc}\right)\right],
\end{align}
where from we obtain
\begin{align}
\sum_{n=0}^\infty\Uxp{{n+1}}{\pv}\tn{n}&=\frac{1}{\pa}\cdot\frac{\pexp_\frac{\pb}{\pa}\left[\pa x\left(\frac{1}{\pc}\pexp_\frac{\pd}{\pc}(\pc t)-\frac{1}{\pc}\right)\right]}{1-\frac{\pb}{\pa}\left[\pa x\left(\frac{1}{\pc}\pexp_\frac{\pd}{\pc}(\pc t)-\frac{1}{\pc}\right)\right]}\cdot\pa x\part{}{t}\left(\frac{1}{\pc}\pexp_\frac{\pd}{\pc}(\pc t)-\frac{1}{\pc}\right).\label{Uderivt1}
\end{align}
Applying \eqref{GF:eab} and performing some simplifications lead to
\begin{align}
\sum_{n=0}^\infty\Uxp{{n+1}}{\pv}\tn{n}&=
\frac{\frac{1}{\pa}\pexp_\frac{\pb}{\pa}
\left[\pa x\left(\frac{1}{\pc}\pexp_\frac{\pd}{\pc}(\pc t)
-\frac{1}{\pc}\right)\right]-\frac{1}{\pa}+\frac{\delta_{\pa,1}}{\pa}
+\frac{1}{\pa}-\frac{\delta_{\pa,1}}{\pa}}{1-\pb x\left(\frac{1}{\pc}
\pexp_\frac{\pd}{\pc}(\pc t)-\frac{1}{\pc}\right)}\nn\\
&\cdot\pa x\part{}{t}\left(\frac{1}{\pc}
\sum_{n=1}^\infty\pc^n\W_n\left(\frac{\pd}{\pc}\right)\tn{n}\right).
\end{align}
After differentiating and applying \eqref{Ffun} and \eqref{GF:eab} we get
\begin{align}
\sum_{n=0}^\infty\Uxp{{n+1}}{\pv}\tn{n}&=
\frac{a\Ffun(x,t;\pv)
+1-\delta_{\pa,1}}{1-\pb x\left(\frac{1}{\pc}
\sum_{n=1}^\infty\pc^n\W_n\left(\frac{\pd}{\pc}\right)\tn{n}\right)}
\cdot x
\sum_{n=0}^\infty\pc^{n}\W_{n+1}\left(\frac{\pd}{\pc}\right)\tn{n}.
\label{hw34}
\end{align}
Multiplying both sides by $1-\pb x\left(\frac{1}{\pc}
\sum_{n=1}^\infty\pc^n\W_n\left(\frac{\pd}{\pc}\right)\tn{n}\right)$ and applying \eqref{GF:UGF} give
\begin{align}
\sum_{n=0}^\infty\Uxp{{n+1}}{\pv}\tn{n}&\cdot\left(1-\pb x\left(\frac{1}{\pc}
\sum_{n=1}^\infty\pc^n\W_n\left(\frac{\pd}{\pc}\right)\tn{n}\right)\right)\nn\\
&=
\left(\pa\sum_{n=0}^\infty\Uxp{{n}}{\pv}\tn{n}
+1-\delta_{\pa,1}\right)\cdot x
\sum_{n=0}^\infty\pc^{n}\W_{n+1}\left(\frac{\pd}{\pc}\right)\tn{n}.
\end{align}
After opening parenthesis and applying the Cauchy product, we get
\begin{align}
\sum_{n=0}^\infty\Uxp{{n+1}}{\pv}\tn{n}&-\pb x\sum_{n=1}^\infty
\sum_{m=0}^{n-1}\binom{n}{m}\Uxp{{m+1}}{\pv}\pc^{n-m-1}\W_{n-m}\left(\frac{\pd}{\pc}\right)\tn{n}\nn\\
&=
\pa x\sum_{n=0}^\infty\sum_{m=0}^n\binom{n}{m}\Uxp{{m}}{\pv}\pc^{n-m}\W_{n-m+1}\left(\frac{\pd}{\pc}\right)\tn{n}\nn\\
&+(1-\delta_{\pa,1}) x
\sum_{n=0}^\infty\pc^{n}\W_{n+1}\left(\frac{\pd}{\pc}\right)\tn{n}.\nn
\end{align}
Extracting coefficients of $\tn{n}$ on both sides gives
\begin{align}
\Uxp{{n+1}}{\pv}&%
-\pb x\sum_{m=0}^{n-1}\binom{n}{m}\Uxp{{m+1}}{\pv}
\pc^{n-m-1}\W_{n-m}\left(\frac{\pd}{\pc}\right)\nn\\
&=
\pa x\sum_{m=0}^n\binom{n}{m}\Uxp{{m}}{\pv}\pc^{n-m}\W_{n-m+1}\left(\frac{\pd}{\pc}\right)\nn\\
&+(1-\delta_{\pa,1}) x
\pc^{n}\W_{n+1}\left(\frac{\pd}{\pc}\right).\nn
\end{align}
By shifting summation indexes under first sum and taking into 
consideration that  $\Uxp{0}{\pv}=\frac{\delta_{\pa,1}}{\pa}$  by Lemma~\ref{lemma1}, 
we obtain
\begin{align}
\Uxp{{n+1}}{\pv}&%
=\pb x\sum_{m=1}^{n}\binom{n}{m-1}\Uxp{{m}}{\pv}
\pc^{n-m}\W_{n-m+1}\left(\frac{\pd}{\pc}\right)\nn\\
&+
ax\sum_{m=1}^n\binom{n}{m}\Uxp{{m}}{\pv}\pc^{n-m}\W_{n-m+1}\left(\frac{\pd}{\pc}\right)\nn\\
&= x\delta_{\pa,1}\pc^n\W_{n+1}\left(\frac{\pd}{\pc}\right)+(1-\delta_{\pa,1}) x
\pc^{n}\W_{n+1}\left(\frac{\pd}{\pc}\right)\nn\\
&=x\sum_{m=1}^n\left[\pa\binom{n}{m}+\pb\binom{n}{m-1}\right]\pc^{n-m}\W_{n-m+1}\left(\frac{\pd}{\pc}\right)\Uxp{m}{\pv}
+x\pc^n\W_{n+1}\left(\frac{\pd}{\pc}\right),
\end{align}
which completes the proof.
\end{proof}

\begin{example} \label{ex2-3}
Applying the result of the previous theorem, we can construct the functions $\Uxp{n}{\pv}$ recursively. One can check that
\begin{align*}
\Uxp{0}{\pv}&=\frac{\delta_{\pa,1}}{\pa},\\
\Uxp{1}{\pv}&=x\W_{1}\left(\frac{\pd}{\pc}\right)=x,\\
\Uxp{2}{\pv}&=x\left(\pa+\pb\right)\W_{1}\left(\frac{\pd}{\pc}\right)\Uxp{1}{\pv}
+x\pc\W_{2}\left(\frac{\pd}{\pc}\right)=(\pa+\pb) x^2+(\pc+\pd) x
\\
\Uxp{3}{\pv}&=x\sum_{m=1}^2\left[\pa\binom{2}{m}+\pb\binom{2}{m-1}\right]\pc^{2-m}\W_{3-m}\left(\frac{\pd}{\pc}\right)\Uxp{m}{\pv}
+x\pc^2\W_{3}\left(\frac{\pd}{\pc}\right)\\
&=(\pa+\pb)(\pa+2\pb) x^3+3(\pa+\pb)(\pc+\pd) x^2
+(\pc+\pd)(\pc+2\pd)x.
\end{align*}
It is easy to notice that $\Uxp{1}{\pv}, \Uxp{2}{\pv}, \Uxp{3}{\pv}$ are actually polynomials in variables $x,\pa,\pb,\pc,\pd$.
\end{example}
The next proposition shows that the last observation is true also for any $\Uxp{n}{\pv}$, where $n$ is a positive integer, denoted further as $n\in\mathbb{N}$.
\begin{prop}\label{prop-degrees}
For all $n\in\mathbb{N}$, $\Uxp{n}{\pv}$ is a polynomial  in the variables  $x$, $\pa$, $\pb$, $\pc$, and $\pd$. Its degrees, respectively, are 
\begin{align}
&deg_x(\Uxp{n}{\pv})=n,\nn\\
&deg_\pa(\Uxp{n}{\pv})=n-1,\nn\\
&deg_\pb(\Uxp{n}{\pv})=n-1,\nn\\
&deg_\pc(\Uxp{n}{\pv})=n-1,\nn\\
&deg_\pd(\Uxp{n}{\pv})=n-1.\nn
\end{align}
\end{prop}
\begin{proof}
The proof is by induction on $n$. By applying \eqref{UGFrec1} from Theorem~\ref{thm2}, one can easy obtain that $\Uxp{1}{\pv}=x$ and the Proposition's statement holds. Assume that it holds for degree $n$, we will prove it for $n+1$. Let us write explicitly expressions for $\W_{n-m+1}\left(\frac{\pd}{\pc}\right)$ and  $\W_{n+1}\left(\frac{\pd}{\pc}\right)$ in \eqref{UGFrec1} to obtain 
\begin{align}
\Uxp{{n+1}}{\pv}&=x\sum_{m=1}^n\left[\pa\binom{n}{m}+\pb\binom{n}{m-1}\right]\pc^{n-m}\prod_{j=1}^{n-m}\left(\frac{j\pd}{\pc}+1\right)\Uxp{m}{\pv} \nn\\&+x\pc^n\prod_{j=1}^{n}\left(\frac{j\pd}{\pc}+1\right),\nn
\end{align}
that can be written as
\begin{align}
\Uxp{{n+1}}{\pv}&=x\sum_{m=1}^n\left[\pa\binom{n}{m}+\pb\binom{n}{m-1}\right]\prod_{j=1}^{n-m}\left(j\pd+\pc\right)\Uxp{m}{\pv} \nn\\
&+x\prod_{j=1}^{n}\left(j\pd+\pc\right).\label{rec1-1}
\end{align}
By induction's assumption $\Uxp{m}{\pv}$ for $m\in\{1,\ldots,n\}$ are polynomials in the variables $x$, $\pa$, $\pb$, $\pc$, and $\pd$, therefore from \eqref{rec1-1} we immediately conclude that $\Uxp{{n+1}}{\pv}$ is a polynomial in the same variables. To complete the proof we have to check the degrees of $\Uxp{{n+1}}{\pv}$ w.r.t. each one of the variables. From \eqref{rec1-1} we get
\begin{align}
&deg_x(\Uxp{{n+1}}{\pv})=\max_m (deg_x(x\Uxp{{m}}{\pv}))=n+1,\nn\\
&deg_\pa(\Uxp{{n+1}}{\pv})=\max_m (deg_\pa(\pa\Uxp{{m}}{\pv}))=(n-1)+1=n,\nn\\
&deg_\pb(\Uxp{{n+1}}{\pv})=\max_m (deg_\pb(\pb\Uxp{{m}}{\pv}))=(n-1)+1=n.\nn
\end{align}
For two other variables we get
\begin{align}
deg_\pc(\Uxp{{n+1}}{\pv})&=\max[\max_m (deg_\pc(\pc^{n-m}\Uxp{{m}}{\pv})),deg_\pc(\pc^n)]\nn\\
&=\max[\max_m((n-m)+(m-1)),n]=n,\nn\\
deg_\pd(\Uxp{{n+1}}{\pv})&=\max[\max_m (deg_\pd(\pd^{n-m}\Uxp{{m}}{\pv})),deg_\pd(\pd^n)]\nn\\
&=\max[\max_m((n-m)+(m-1)),n]=n,\nn
\end{align}
which completes the proof.
\end{proof}

Let us denote by $[x^m]\Uxp{n}{\pv}$ the coefficient of $x^m$ of the polynomial $\Uxp{n}{\pv}$. Then we can state the following Lemma.
\begin{lemma} \label{lemma_zero-term}
For all $n\in\mathbb{N}$, we have $[x^0]\Uxp{n}{\pv}=0$.
\end{lemma}
\begin{proof}
The proof is by induction on $n$. It is easy to obtain from the Theorem~\ref{thm2} and the Lemma~\ref{lemma1} that $\Uxp{1}{\pv}=x$, and our statement holds. Assuming it holds for $n$, one can immediately see form \eqref{UGFrec1} that it holds for $n+1$.
\end{proof}

Another observation that can be done from the Example \ref{ex2-3} is that the coefficients of the highest and lowest degrees of $x$ in each one of the polynomials $\Uxp{{2}}{\pv}, \Uxp{{3}}{\pv}$ are not mixed and depend either on the pair $(\pa,\pb)$ or on the pair $(\pc,\pd)$. The next Corollary of the Theorem \ref{thm2} generalizes this observation for any polynomial $\Uxp{{n}}{\pv}$.  
\begin{corollary} For any $n\in\mathbb{N}$, we have 
\begin{align*}
[x^n]\Uxp{n}{\pv}&=\pa^{n-1}\W_n\left(\frac{\pb}{\pa}\right)
=\prod_{m=1}^{n-1}(\pa+m\pb),\\
[x^1]\Uxp{n}{\pv}&=\pc^{n-1}\W_n\left(\frac{\pd}{\pc}\right)=\prod_{m=1}^{n-1}(\pc+m\pd).
\end{align*}
\end{corollary}
\begin{proof}
It follows from the \eqref{UGFrec1} that
\begin{align*}
[x^1]\Uxp{{n}}{\pv}=\pc^{n-1}\W_n\left(\frac{\pd}{\pc}\right).
\end{align*}
Now, applying the definition \eqref{Wn} of $\W_n$ we obtain the Corollary's second statement.
In order to obtain the coefficient of $x^n$ of the polynomial $\Uxp{n}{\pv}$, let us consider \eqref{UGFrec1} again.
\begin{align*}
[x^n]\Uxp{n}{\pv}&=[x^{n-1}]\left[\pa\binom{n}{n}+\pb\binom{n}{n-1}\right]\pc^{n-n}\W_{n-n+1}\left(\frac{\pd}{\pc}\right)\Uxp{{n-1}}{\pv}\\
&=(\pa+n\pb)\cdot[x^{n-1}]\Uxp{{n-1}}{\pv}\\
&=(\pa+n\pb)(\pa+(n-1)\pb)\cdot[x^{n-2}]\Uxp{{n-2}}{\pv}=\cdots\\
&=\prod_{m=1}^{n-1}(\pa+m\pb),
\end{align*}
which completes the proof.
\end{proof}

Next two lemmas show connections between polynomials $\Uxp{n}{\pv}$ and their partial derivatives w.r.t. variable $x$.
\begin{lemma}\label{l2-3} For all $n\in \mathbb{N}$, we have
\begin{align}
\part{}{x}\Uxp{n}{\pv}&=\pa\sum_{m=1}^{n-1}\binom{n}{m}\pc^{n-m-1}\W_{n-m}\frdc \Uxp{m} {\pv}\label{derx1}+ \pc^{n-1}\W_{n}\frdc\\
&+\pb x\sum_{m=1}^{n-1}\binom{n}{m}\pc^{n-m-1}\W_{n-m}\frdc\part{}{x}\Uxp{m}{\pv}.\nn
\end{align}
\end{lemma}
\begin{proof}
We differentiate \eqref{GF:UGF} w.r.t. $x$ and obtain

\begin{align}
\sum_{n=0}^\infty\part{}{x}\Uxp{n}{\pv}\tn{n}&=
\part{}{x}\left[\frac{1}{\pa}\pexp_\frac{\pb}{\pa}\left[\pa x
\left(\frac{1}{\pc}\pexp_\frac{\pd}{\pc}(\pc t)-\frac{1}{\pc}\right)\right]
-\frac{1}{\pa}+\frac{\delta_{\pa,1}}{\pa}\right]\nn\\
&=\frac{1}{\pa}\part{}{x}\pexp_\frac{\pb}{\pa}\left[\pa x\left(\frac{1}{\pc}
\pexp_\frac{\pd}{\pc}(\pc t)-\frac{1}{\pc}\right)\right],\nn
\end{align}
where from we obtain
\begin{align}
\sum_{n=0}^\infty\part{}{x}\Uxp{n}{\pv}\tn{n}&=
\frac{1}{\pa}\cdot\frac{\pexp_\frac{\pb}{\pa}
\left[\pa x\left(\frac{1}{\pc}\pexp_\frac{\pd}{\pc}(\pc t)
-\frac{1}{\pc}\right)\right]}{1-\frac{\pb}{\pa}
\left[\pa x\left(\frac{1}{\pc}\pexp_\frac{\pd}{\pc}(\pc t)
-\frac{1}{\pc}\right)\right]}
\cdot\pa\left(\frac{1}{\pc}\pexp_\frac{\pd}{\pc}(\pc t)-\frac{1}{\pc}\right).\label{Uderivx1}
\end{align}
By applying \eqref{GF:bexp} and \eqref{Ffun}, we write \eqref{Uderivx1} as following.
\begin{align}
\sum_{n=0}^\infty\part{}{x}\Uxp{n}{\pv}\tn{n}&=
\frac{\pa \Ffun +1-\delta_{\pa,1}}{1-\pb
 x\sum_{n=1}^\infty \pc^{n-1}\W_{n}\frdc\tn{n}}
\cdot\sum_{n=1}^\infty \pc^{n-1}\W_{n}\frdc\tn{n}.\label{hw44}
\end{align}
After multiplication both sides by $1-\pb x
 \sum\limits_{n=1}^\infty \pc^{n-1}\W_{n}\frdc\tn{n}$ and applying \eqref{GF:UGF}, we get
 \begin{align}
\sum_{n=0}^\infty\part{}{x}\Uxp{n}{\pv}\tn{n}
&\cdot
\left(1-\pb x\sum_{n=1}^\infty \pc^{n-1}\W_{n}\frdc\tn{n}\right)\nn\\
&=
\left(\pa \sum_{n=0}^\infty\Uxp{n}{\pv}\tn{n}+1-\delta_{\pa,1}\right)
\cdot\sum_{n=1}^\infty \pc^{n-1}\W_{n}\frdc\tn{n}.
\end{align}
Opening parenthesis and applying the Cauchy product lead to 
\begin{align}
\sum_{n=0}^\infty&\part{}{x}\Uxp{n}{\pv}\tn{n}-\pb x\sum_{n=1}^\infty\sum_{m=0}^{n-1}\binom{n}{m}\part{}{x}\Uxp{m}{\pv}\pc^{n-m-1}\W_{n-m}\frdc\tn{n}\nn\\
&=\pa\sum_{n=1}^\infty\sum_{m=0}^{n-1}\binom{n}{m}\Uxp{m}{\pv}\pc^{n-m-1}\W_{n-m}\frdc\tn{n}
+(1-\delta_{\pa,1})\sum_{n=1}^\infty \pc^{n-1}\W_{n}\frdc\tn{n},\nn
\end{align}
where from, by extracting the coefficients of $\tn{n}$, we obtain for positive integers $n$ that
\begin{align}
\part{}{x}\Uxp{n}{\pv}&-\pb x\sum_{m=0}^{n-1}\binom{n}{m}\part{}{x}\Uxp{m}{\pv}\pc^{n-m-1}\W_{n-m}\frdc\nn\\
&=\pa\sum_{m=0}^{n-1}\binom{n}{m}\Uxp{m}{\pv}\pc^{n-m-1}\W_{n-m}\frdc
+(1-\delta_{\pa,1}) \pc^{n-1}\W_{n}\frdc.\nn
\end{align}
Applying the fact that $\Uxp{0}{\pv}=\frac{\delta_{\pa,1}}{\pa}$ and rearranging of the terms complete the proof.
\end{proof}

\begin{lemma}\label{l2-4} For all $n\in \mathbb{N}$, we have
\begin{align}\label{derx2}
\sum_{m=0}^{n-1}\binom{n}{m}\pc^{n-m-1}&\W_{n-m}\frdc\Uxp{{m+1}}{\pv}\\
&=x\sum_{m=1}^n\binom{n}{m}\pc^{n-m}\W_{n-m+1}\frdc\part{}{x}\Uxp{m}{\pv}.\nn
\end{align}
\end{lemma}
\begin{proof}
After comparing equations \eqref{hw34} and \eqref{hw44}, we multiply them respectively by $\sum_{n=1}^\infty\pc^{n-1}\W_n\frdc\tn{n}$ 
 by $x\sum_{n=0}^\infty\pc^n\W_{n+1}\frdc\tn{n}$. Thus we obtain
 \begin{align}
 \sum_{n=0}\Uxp{{n+1}}{\pv}\tn{n}\cdot\sum_{n=1}^\infty\pc^{n-1}\W_n\frdc\tn{n}=\sum_{n=0}^\infty\part{}{x}\Uxp{n}{\pv}\tn{n}\cdot x\sum_{n=0}^\infty\pc^n\W_{n+1}\frdc\tn{n},\nn
 \end{align}
 which leads to
\begin{align}
\sum_{n=1}^\infty\sum_{m=0}^{n-1}\binom{n}{m}\pc^{n-m-1}&\W_{n-m}\frdc\Uxp{{m+1}}{\pv}\tn{n}\nn\\
&=\sum_{n=0}^\infty\sum_{m=0}^{n}\binom{n}{m}x\pc^{n-m}\W_{n-m+1}\frdc\part{}{x}\Uxp{m}{\pv}.
\end{align}
Extracting the coefficients of $\tn{n}$ on both sides completes the proof.
\end{proof}
We are ready to state the following theorem.
\begin{theorem} 
For all $n\in\mathbb{N}$, the polynomials $\Uxp{n}{\pv}$ satisfy the following recurrence relation
\begin{align}
\Uxp{{n+1}}{\pv}=(\pa x+\pd n)\Uxp{n}{\pv}+x(\pb x +\pc)\part{}{x}\Uxp{n}{\pv},\label{mainrec}
\end{align}
with the initial condition $\Uxp{1}{\pv}=x$.
\end{theorem}
\begin{proof}
Let us consider \eqref{Uderivt1}, where we obtained that
\begin{align}
\sum_{n=0}^\infty\Uxp{{n+1}}{\pv}\tn{n}&=\frac{1}{\pa}\cdot\frac{\pexp_\frac{\pb}{\pa}\left[\pa x\left(\frac{1}{\pc}\pexp_\frac{\pd}{\pc}(\pc t)-\frac{1}{\pc}\right)\right]}{1-\frac{\pb}{\pa}\left[\pa x\left(\frac{1}{\pc}\pexp_\frac{\pd}{\pc}(\pc t)-\frac{1}{\pc}\right)\right]}\cdot\pa x\part{}{t}\left(\frac{1}{\pc}\pexp_\frac{\pd}{\pc}(\pc t)-\frac{1}{\pc}\right).\nn
\end{align}
By applying  \eqref{GF:bexp}-\eqref{Wn}, \eqref{bexp-deriv}, and \eqref{Ffun}, we get
\begin{align}
\sum_{n=0}^\infty\Uxp{{n+1}}{\pv}\tn{n}&=\frac{\pa\Ffun(x,t;\pv)+1-\delta_{\pa,1}}{1-\frac{\pb}{\pa}\left[\pa x\left(\frac{1}{\pc}\sum_{n=1}^\infty\pc^n\W_{n}\left(\frac{\pd}{\pc}\right)\tn{n}\right)\right]}\cdot \pc x\frac{\frac{1}{\pc}\pexp_\frac{\pd}{\pc}(\pc t)}{1-\frac{\pd}{\pc}\cdot\pc t}
\nn\\
&=\frac{\pa\Ffun(x,t;\pv)+1-\delta_{\pa,1}}{1-\pb x\sum_{n=1}^\infty\pc^{n-1}\W_{n}\left(\frac{\pd}{\pc}\right)\tn{n}}\cdot  x\frac{\pexp_\frac{\pd}{\pc}(\pc t)}{1-\pd  t}.\nn
\end{align}
Applying again \eqref{GF:bexp}-\eqref{Wn} gives that
\begin{align}
\sum_{n=0}^\infty\Uxp{{n+1}}{\pv}\tn{n}&=\frac{\pa\Ffun(x,t;\pv)+1-\delta_{\pa,1}}{1-\pb x\sum_{n=1}^\infty\pc^{n-1}\W_{n}\left(\frac{\pd}{\pc}\right)\tn{n}}\cdot  x\frac{\sum_{n=0}^\infty\pc^{n}\W_{n}\left(\frac{\pd}{\pc}\right)\tn{n}}{1-\pd  t}.\label{p2005-eq1}
\end{align}
By multiplying \eqref{p2005-eq1} by 
$\left(1-\pb x\sum\limits_{n=1}^\infty\pc^{n-1}\W_{n}\left(\dfrac{\pd}{\pc}\right)
\displaystyle{\tn{n}}\right)(1-\pd  t)$ and substituting \eqref{GF:UGF}, we obtain
\begin{align}
\sum_{n=0}^\infty\Uxp{{n+1}}{\pv}\tn{n}&\cdot\left(1-\pb x\sum_{n=1}^\infty\pc^{n-1}\W_{n}\left(\frac{\pd}{\pc}\right)\tn{n}\right)(1-\pd  t)\nn\\
&=\left(\pa\sum_{n=0}^\infty\Uxp{n}{\pv}\tn{n}+1-\delta_{\pa,1}\right)\cdot  x\sum_{n=0}^\infty\pc^{n}\W_{n}\left(\frac{\pd}{\pc}\right)\tn{n}.\nn
\end{align}
After opening parenthesis and applying the Cauchy product, we get
\begin{align} \label{p2005-eq3}
\sum_{n=0}^\infty\Uxp{{n+1}}{\pv}\tn{n}&
-\pb x\sum_{n=1}^\infty\sum_{m=0}^{n-1}\binom{n}{m}\Uxp{{m+1}}{\pv}\pc^{n-m-1}\W_{n-m}\frdc\tn{n}\\%\nn\\
&-\pd t\sum_{n=0}^\infty\Uxp{{n+1}}{\pv}\tn{n}\nn\\
&+\pb\pd xt\sum_{n=1}^\infty\sum_{m=0}^{n-1}\binom{n}{m}\Uxp{{m+1}}{\pv}\pc^{n-m-1}\W_{n-m}\frdc\tn{n}\nn\\
&=\pa x \sum_{n=0}^\infty\sum_{m=0}^{n}\binom{n}{m}\Uxp{m}{\pv}\pc^{n-m}\W_{n-m}\frdc\tn{n}\nn\\
& +(1-\delta_{\pa,1})x\sum_{n=0}^\infty\pc^{n}\W_{n}\frdc\tn{n}.\nn
\end{align}
By extracting the coefficients of $\tn{n}$ on both sides and rearranging the terms, we obtain for $n\geq 1$ that
\begin{align}\label{hw56-57}
\Uxp{{n+1}}{\pv}&=
\pb x\sum_{m=0}^{n-1}\binom{n}{m}\Uxp{{m+1}}{\pv}\pc^{n-m-1}\W_{n-m}\frdc+\pd n\Uxp{n}{\pv}\\
&-\pb\pd nx\sum_{m=0}^{n-2}\binom{n-1}{m}\Uxp{{m+1}}{\pv}\pc^{n-m-2}\W_{n-m-1}\frdc\nn\\
&+\pa x \sum_{m=0}^{n}\binom{n}{m}\Uxp{m}{\pv}\pc^{n-m}\W_{n-m}\frdc \nn\\
& +(1-\delta_{\pa,1})x\pc^{n}\W_{n}\frdc.\nn
\end{align}
Let us consider the terms on the third and fourth lines of \eqref{hw56-57}. We have
\begin{align}\nn
\Uxp{{n+1}}{\pv}&=
\pb x\sum_{m=0}^{n-1}\binom{n}{m}\Uxp{{m+1}}{\pv}\pc^{n-m-1}\W_{n-m}\frdc+\pd n\Uxp{n}{\pv}\\
&-\pb\pd nx\sum_{m=0}^{n-2}\binom{n-1}{m}\Uxp{{m+1}}{\pv}\pc^{n-m-2}\W_{n-m-1}\frdc\nn\\
&+\pc x \Big[\part{}{x}\Uxp{n}{\pv}-\pb x\sum_{m=0}^{n-1}\binom{n}{m}\part{}{x}\Uxp{m}{\pv}\pc^{n-m-1}\W_{n-m}\frdc\Big] \nn\\
&+\pa x\Uxp{n}{\pv},
\nn
\end{align}
which can be written as
\begin{align}\label{hw59}
\Uxp{{n+1}}{\pv}&=\pa x\Uxp{n}{\pv}+\pc x\part{}{x}\Uxp{n}{\pv}+\pd n\Uxp{n}{\pv}\\
&-\pb\pc x^2\sum_{m=0}^{n-1}\binom{n}{m}\pc^{n-m-1}\W_{n-m}\frdc\part{}{x}\Uxp{m}{\pv}\nn\\
&+\pb x\sum_{m=0}^{n-1}\binom{n}{m}\Uxp{{m+1}}{\pv}\pc^{n-m-1}\W_{n-m}\frdc\nn\\
&-\pb\pd nx\sum_{m=0}^{n-2}\binom{n-1}{m}\Uxp{{m+1}}{\pv}\pc^{n-m-2}\W_{n-m-1}\frdc.\nn
\end{align}
We apply Lemma~\ref{l2-4} to the term on the third line of \eqref{hw59} and obtain
\begin{align}\label{hw61}
\Uxp{{n+1}}{\pv}&=\pa x\Uxp{n}{\pv}+\pc x\part{}{x}\Uxp{n}{\pv}+\pd n\Uxp{n}{\pv}\\
&-\pb x^2\sum_{m=0}^{n-1}\binom{n}{m}\pc^{n-m}\W_{n-m}\frdc\part{}{x}\Uxp{m}{\pv}\nn\\
&+\pb x^2\sum_{m=0}^{n}\binom{n}{m}\part{}{x}\Uxp{{m}}{\pv}\pc^{n-m}\W_{n-m+1}\frdc\nn\\
&-\pb\pd nx\sum_{m=0}^{n-2}\binom{n-1}{m}\Uxp{{m+1}}{\pv}\pc^{n-m-2}\W_{n-m-1}\frdc.\nn
\end{align}
Let us consider the terms on the second and the third lines of \eqref{hw61}.  By applying \eqref{Wn} and extracting a common factor, their sum can be written as 
\begin{align}
\pb x^2\sum_{m=0}^{n-1}\binom{n}{m}&\pc^{n-m}\W_{n-m}\frdc\part{}{x}\Uxp{m}{\pv}\cdot\left[\frac{(n-m)\pd}{\pc}+1-1\right]+\pb x^2\part{}{x}\Uxp{n}{\pv}\nn\\
&=\pb x^2\sum_{m=0}^{n-1}\binom{n}{m}\pc^{n-m-1}\W_{n-m}\frdc\part{}{x}\Uxp{m}{\pv}\cdot(n-m)\pd\nn\\
&+\pb x^2\part{}{x}\Uxp{n}{\pv}\nn\\
&=\pb \pd nx^2\sum_{m=0}^{n-1}\binom{n-1}{m}\pc^{n-m-1}\W_{n-m}\frdc\part{}{x}\Uxp{m}{\pv}+\pb x^2\part{}{x}\Uxp{n}{\pv}.\nn
\end{align}
We substitute this expression into \eqref{hw61} and obtain
\begin{align}\label{hw64}
\Uxp{{n+1}}{\pv}&=\pa x\Uxp{n}{\pv}+\pc x\part{}{x}\Uxp{n}{\pv}+\pd n\Uxp{n}{\pv}+\pb x^2\part{}{x}\Uxp{n}{\pv}\\ 
&+\pb \pd nx^2\sum_{m=0}^{n-1}\binom{n-1}{m}\pc^{n-m-1}\W_{n-m}\frdc\part{}{x}\Uxp{m}{\pv}\nn\\ 
&-\pb\pd nx\sum_{m=0}^{n-2}\binom{n-1}{m}\Uxp{{m+1}}{\pv}\pc^{n-m-2}\W_{n-m-1}\frdc.\nn 
\end{align}
Finally, applying the Lemma~\ref{l2-4} to  the term on the second line of \eqref{hw64}, we obtain
\begin{align}\label{hw64}
\Uxp{{n+1}}{\pv}&=\pa x\Uxp{n}{\pv}+\pc x\part{}{x}\Uxp{n}{\pv}+\pd n\Uxp{n}{\pv}+\pb x^2\part{}{x}\Uxp{n}{\pv}\\ 
&+\pb \pd nx\sum_{m=0}^{n-2}\binom{n-1}{m}\pc^{n-m-2}\W_{n-m-1}\frdc\Uxp{m+1}{\pv}\nn\\
&-\pb\pd nx\sum_{m=0}^{n-2}\binom{n-1}{m}\Uxp{{m+1}}{\pv}\pc^{n-m-2}\W_{n-m-1}\frdc,\nn
\end{align}
and an obvious simplification completes the proof.
\end{proof}

We turn to consider the coefficients of the polynomials $\Uxp{n}{\pv}$. 
Let $\Uc{n}{m}{\pv}:=[x^m]\Uxp{n}{\pv}$, so we have
\begin{align}
\Uxp{n}{\pv}=\sum_{m=1}^n\Uc{n}{m}{\pv}x^m.\label{Upol}
\end{align}
Note, that $\Uc{n}{m}{\pv}$ is a polynomial in the variables $\pa$, $\pb$, $\pc$, and $\pd$. Moreover, for consistency, we define $\Uc{0}{0}{\pv}:=\Uxp{0}{\pv}=\frac{\delta_{\pa,1}}{\pa}$.
\begin{prop}\label{prop2-9}
For all $n, m\in\mathbb{N}$, we have
\begin{align}
\Uc{{n+1}}{m}{\pv}=\left\{\begin{array}{ll}
(\pc+n\pd)\Uc{n}{1}{\pv}, & m=1,\\
(\pa+(m-1)\pb)\Uc{n}{{m-1}}{\pv}+(m\pc+n\pd)\Uc{n}{m}{\pv}, & 2\leq m\leq n,\\
(\pa+n\pb)\Uc{n}{n}{\pv}, & m=n+1,\\
0, & \text{otherwise},
\end{array}\right.
\end{align}
with initial conditions $\Uc{0}{0}{\pv}=\frac{\delta_{\pa,1}}{\pa}$, $\Uc{1}{1}{\pv}=1$.
\end{prop}
\begin{proof}
We defined $\Uc{{n+1}}{m}{\pv}$ as the coefficients of the polynomial $\Uxp{{n+1}}{\pv}$. We proved in the Proposition~\ref{prop-degrees} that $\Uxp{{n+1}}{\pv}$ is a polynomial of degree $n+1$ in the variable $x$. Moreover, in the Lemma~\ref{lemma_zero-term}, we show that the coefficient of $x^0$ is 0 for all $n\in\mathbb{N}$.Therefore one can immediately conclude that  $\Uc{{n+1}}{m}{\pv}=0$ in case $m\notin\{1,\ldots,n+1\}$.  
 Let us substitute \eqref{Upol} into \eqref{mainrec}.
\begin{align}
\sum_{m=1}^{n+1}\Uc{{n+1}}{m}{\pv}x^m&=(\pa x+\pd n)\sum_{m=1}^n\Uc{n}{m}{\pv}x^m+x(\pb x +\pc)\part{}{x}\sum_{m=1}^n\Uc{n}{m}{\pv}x^m\nn\\
&=(\pa x+\pd n)\sum_{m=1}^n\Uc{n}{m}{\pv}x^m+(\pb x +\pc)\sum_{m=1}^nm\Uc{n}{m}{\pv}x^m.\nn
\end{align}
Extracting the coefficients of $x^m$ on both sides gives for $m=1$
\begin{align}
\Uc{{n+1}}{1}{\pv}=\pd n\Uc{n}{m}{\pv}+\pc \Uc{n}{m}{\pv},\nn
\end{align}
for $2\leq m\leq n$
\begin{align}
\Uc{{n+1}}{m}{\pv}=\pa \Uc{n}{m-1}{\pv}+\pd n\Uc{n}{m}{\pv}+\pb  (m-1)\Uc{n}{{m-1}}{\pv}+\pc m\Uc{n}{m}{\pv},\nn
\end{align}
and for $m=n+1$
\begin{align}
\Uc{{n+1}}{{n+1}}{\pv}=\pa \Uc{n}{n}{\pv}+\pb n\Uc{n}{n}{\pv}.\nn
\end{align}
Rearrangement of the terms completes the proof.
\end{proof}
The generating function $\Ffun(x,t;\pv)$ defined by equation \eqref{GF:UGF}, satisfies the following differential equations.
\begin{theorem} Let $n\in\mathbb{N}$. Then the differential equation
\begin{align}
f^{(n)}=\sum_{m=1}^n \Uc{{n}}{m}{\pv}\pa^{1+m}x^m(1-\pd t)^{-m\frac{\pc}{\pd} -n}\left(f+\frac{1}{\pa}-\frac{\delta_{\pa,1}}{\pa}\right)^{1+m\frac{\pb}{\pa}},
\end{align}
where $f^{(n)}=\frac{\partial^n}{\partial t^n}f$, $\pv=(\pa,\pb,\pc,\pd)$, has a solution $f=\Ffun(x,t;\pv)$.
\end{theorem}
\begin{proof}
Let $f=\Ffun(x,t;\pa,\pb,\pc,\pd)$. The proof is by induction on $n$. Differentiating w.r.t. $t$ and applying the definitions \eqref{GF:UGF}, \eqref{GF:bexp}, and the property \eqref{bexp-deriv}, we obtain that
\begin{align}\label{first-der}
\Ffun^{(1)}=\pa^2x(1-\pd t)^{-\frac{\pc}{\pd}-1}\left(F+\frac{1}{\pa}-\frac{\delta_{\pa,1}}{\pa}\right)^{1+\frac{\pb}{\pa}},
\end{align}
and, since $\Uc{1}{1}{\pv}=1$, we obtain the Theorem's statement. Assume that the $(n)$th derivative has the form of
\begin{align*}
\Ffun^{(n)}=\sum_{m=1}^n\Uc{n}{m}{\pv}\pa^{1+m}x^m(1-\pd t)^{-m\frac{\pc}{\pd}-n}\left(\Ffun+\frac{1}{\pa}-\frac{\delta_{\pa,1}}{\pa}\right)^{1+m\frac{\pb}{\pa}}.
\end{align*} 
Let us consider the $(n+1)$th derivative.
\begin{align*}
\Ffun^{(n+1)}&=\sum_{m=1}^n\Uc{n}{m}{\pv}\pa^{1+m}x^m\pd\left(m\frac{\pc}{\pd}+n\right)(1-\pd t)^{-m\frac{\pc}{\pd}-n}\left(\Ffun+\frac{1}{\pa}-\frac{\delta_{\pa,1}}{\pa}\right)^{1+m\frac{\pb}{\pa}}\\
&+\sum_{m=1}^n\Uc{n}{m}{\pv}\pa^{1+m}x^m(1-\pd t)^{-m\frac{\pc}{\pd}-n}\left(\Ffun+\frac{1}{\pa}-\frac{\delta_{\pa,1}}{\pa}\right)^{m\frac{\pb}{\pa}}\left(1+m\frac{\pb}{\pa}\right)\Ffun^{(1)}\\
&=\sum_{m=1}^n(m\pc+n\pd)\Uc{n}{m}{\pv}\pa^{1+m}x^m(1-\pd t)^{-m\frac{\pc}{\pd}-(n+1)}\left(\Ffun+\frac{1}{\pa}-\frac{\delta_{\pa,1}}{\pa}\right)^{1+m\frac{\pb}{\pa}}\\
&+\sum_{m=1}^n(\pa+m\pb)\Uc{n}{m}{\pv}\pa^{2+m}x^{m+1}(1-\pd t)^{-(m+1)\frac{\pc}{\pd}-n}\left(\Ffun+\frac{1}{\pa}-\frac{\delta_{\pa,1}}{\pa}\right)^{1+(m+1)\frac{\pb}{\pa}}\\
&=\sum_{m=1}^n(m\pc+n\pd)\Uc{n}{m}{\pv}\pa^{1+m}x^m(1-\pd t)^{-m\frac{\pc}{\pd}-(n+1)}\left(\Ffun+\frac{1}{\pa}-\frac{\delta_{\pa,1}}{\pa}\right)^{1+m\frac{\pb}{\pa}}\\
&+\sum_{m=2}^{n+1}(\pa+(m-1)\pb)\Uc{n}{m-1}{\pv}\pa^{1+m}x^{m}(1-\pd t)^{-m\frac{\pc}{\pd}-n}\left(\Ffun+\frac{1}{\pa}-\frac{\delta_{\pa,1}}{\pa}\right)^{1+m\frac{\pb}{\pa}}\\
&=\sum_{m=1}^na_{n+1,m}\pa^{1+m}x^m(1-\pd t)^{-m\frac{\pc}{\pd}-n}\left(\Ffun+\frac{1}{\pa}-\frac{\delta_{\pa,1}}{\pa}\right)^{1+m\frac{\pb}{\pa}},
\end{align*}
where, by comparing to the recurrence relation from the Proposition \ref{prop2-9},
\begin{align*}
a_{n+1,m}=(m\pc+n\pd)\Uc{n}{m}{\pv}+(\pa+(m-1)\pb)\Uc{n}{m-1}{\pv}\equiv\Uc{n+1}{m}{\pv}, 
\end{align*}
which completes the proof.
\end{proof}
\begin{remark}
This Theorem generalizes the results of the Theorem 1.1 from \cite{Kim}.
\end{remark}

\section{Combinatorics of set partitions}\label{Section-Comb}
In this section we consider different kinds of set partitions, define two new statistics for them, and study their connections to the polynomials defined in the previous section. We start from some definitions.

An unordered  partition of a set $[n]$ into $k$ blocks is a set of nonempty disjoint subsets $B_1, \ldots, B_k\subseteq[n]$, whose union is $[n]$. A minimal element of a block is called an \emph{opener} (standard opener). Partition can be ordered. In his work  \cite{Motzkin1971}, Motzkin defined names for four different kinds of ordering of set partitions, namely, sets of sets, sets of lists, lists of sets, lists of lists. In such definition, set is standing for an unordered collection of elements, while list is an ordered collection of elements. A generalized exponential generating function defined and studied in \cite{Herscovici2017, Herscovici2019} unified those four kinds of set partitions. Combinatorially, the generalization was based on two statistics proposed in \cite{Herscovici2017}. However, there are more kinds of set partitions.

Here we consider the partitions of kind ``list of lists'' (see \cite{Motzkin1971}), where both the blocks and the elements inside of them are ordered.
Moreover, we will consider four statistics to classify set partitions. Two of them, namely, $\nsb$ and $\nse$ were defined in \cite{Herscovici2017}. Considering a sequence of blocks openers, statistic $\nsb$ says how many blocks should be moved to the right in order to get the increasing sequence of openers. Statistic $\nse$ considers sequences of elements inside the blocks and says how many, in total, elements should be moved inside the blocks in order to get inside each block an increasing sequence of its elements. Additionally to them, we define two new statistics -- $\rlb$ and $\rle$, which will be described below. It was shown in \cite{Herscovici2019} that the statistic $\nse$ is closely related to the statistic counting right-to-left-minima. Given permutation $\sigma=\sigma_1\sigma_2\cdots\sigma_n$ of the set $[n]$. An element $\sigma_j$ is called right-to-left minimum if $\sigma_j<\sigma_m$ for all $m>j$. Obviously, the element $\sigma_n$ is always a right-to-left minimum. Moreover, assuming $\sigma_j=1$, we have that  there is no right-to-left minima among $\sigma_1,\sigma_2,\ldots,\sigma_{j-1}$. Therefore, we define the statistic $\rle$ (``right-to-left elements'') as following.
\begin{align}
\rle(\sigma)=|\{j\big|(\sigma_j<\sigma_m \text{ for all } m>j)\text{ and } (\sigma_j>1)\}|.\label{rle-stat}
\end{align}
Given partition $\pi=B_1/B_2/\cdots/B_k$ of the set $[n]$ into $k$ blocks. We can consider an ordered sequence of its blocks' openers as a permutation. The define the statistic $\rlb$ (``right-to-left blocks'') as following.
\begin{align}
\rlb(\pi)=|\{j\big|(\min B_j<\min B_m \text{ for all } m>j)\text{ and } (\min B_j>1)\}|,\label{rlb-stat}
\end{align}
where by $\min B_m$ we denote the minimal element of the block $B_m$, i.e. the opener of the block $m$. The statistic $\rle$ can be adopted for partitions, where the elements of a block are considered as a permutation and the minimal element of permutation (which is one in standard permutation) is the minimal element of the block, so given $B_\ell=\sigma_1\cdots\sigma_{n_\ell}$ we are able to evaluate $\rle(B_\ell)$ as following
\begin{align}
\rle(B_\ell)=|\{j\big|(\sigma_j<\sigma_m \text{ for all } m>j)\text{ and } (\sigma_j>\min (B_\ell))\}|,\label{rle-block}
\end{align}
 and $\rle(\pi)=\sum_{j=1}^k\rle(B_j)$. Note that $\nsb(\pi)+\rlb(\pi)=k-1$, while $\nse(\pi)+\rle(\pi)=n-k$.

\begin{example} Let us consider a set partition $\pi=382/147/96/5$. Its openers are $2,1,6,5$. This sequence will be standard ordered if the first and the third blocks will be moved to the right as following $\pi=147/382/5/96$, which means that $\nsb(\pi)=2$. The opener $5$ contributes also to the statistic $\rlb$, so we have $\rlb(\pi)=1$. Now let us consider the elements inside the blocks. In order to produce strictly increasing sequences of elements inside the blocks the elements $3,8,9$ should be moved to the right as following $\pi=238/147/69/5$, which means that $\nse(\pi)=3$. Let us evaluate the value of the statistic $\rle$ on the partition $\pi$. The right-to-left elements are $7,4$, which means $\rle(\pi)=2$.
\end{example}

Let us denote by $\LLP_{n,m}$ a set of all list of lists partitions of the set $[n]$ into $m$ blocks, and by 
$\Sfun(n,m;\pv)$ a generating function for partitions $\pi\in\LLP_{n,m}$ w.r.t. the statistics $\nsb$, $\nse$, $\rlb$, and $\rle$, that is
\begin{align}
\Sfun(n,m;\pv)=\sum_{\pi\in\LLP_{n,m}}\pa^{\rlb(\pi)}\pb^{\nsb(\pi)}\pc^{\rle(\pi)}\pd^{\nse(\pi)}.\label{Sfun}
\end{align}
Obviously, we have $\Sfun(1,1;\pv)=1$. Moreover, it follows immediately from \eqref{Sfun} that $\Sfun(n,m;\pv)$ is defined 
to be a polynomial in variables $\pa,\pb,\pc,\pd$ with positive coefficients for all positive integer $n$ with $1\leq m\leq n$ and
zero otherwise. 

\begin{theorem} \label{thm4-2}
For all $m,n\in\mathbb{N}$ such that $n\geq m\geq 1$, 
we have
\begin{align}
\Sfun(n,m;\pv)=
(\pa+(m-1)\pb)\Sfun(n-1,m-1;\pv)+(m\pc+(n-1)\pd)\Sfun(n-1,m;\pv),
\label{Sfunrec} 
\end{align}
with an initial condition $\Sfun(1,1;\pv)=1$.
\end{theorem}

\begin{proof}
The initial condition is easily obtained from \eqref{Sfun} by considering the only partition of the set $[1]$. For such partition $\pi$ we have $\rlb(\pi)=\nsb(\pi)=\rle(\pi)=\rlb(\pi)=0$, and so $\Sfun(1,1,\pv)=1$.
To obtain $\Sfun(n,m;\pv)$ for some $n$, $m$ we need to partition the set $[n]$ into $m$ blocks. Let us assume that elements of the set $[n-1]\subset[n]$ are already partitioned into ordered blocks and we need to arrange only the maximal element $n$.  There are two possibilities: $[n-1]$ is partitioned either into $m$ blocks or into $m-1$ blocks.

If we already have $m$ blocks then the element $n$ should be inserted into one of the existing blocks. Since $n$ is the maximal element of the set $[n]$, such insertion will not change the openers of the blocks and, thus will not change statistics $\nsb$ and $\rlb$. 
Inserting $n$ before any element inside a block will increase statistic $\nse$ by 1, and there are $n-1$ possible places. From another side, since $n$ is maximal, such insertion will not change the statistic $\rle$. If we insert $n$ as the last element of some block, then the statistic $\nse$ is not changed, while the statistic $\rle$ is increased by 1. We have exactly $m$ places to insert $n$ as the last element of a block. The number of partitions of $[n-1]$ into $m$ blocks w.r.t. the statistics $\nsb$, $\nse$, $\rlb$, and $\rle$ is given by $\Sfun(n-1,m;\pv)$. Summarizing this case, we obtain 
$(m\pc+(n-1)\pd)\Sfun(n-1,m;\pv)$.

If the set $[n-1]$ is partitioned into $m-1$ blocks then the element $n$ should be inserted as a new block whose opener is $n$ itself, and, thus, it will be the maximal opener among all $m$ openers.
Inserting block containing $n$ before any of the existing $m-1$ blocks will increase the statistic $\nsb$ by 1, but will not change the statistic $\rlb$. We have $m-1$ possible positions for such insertion. On another hand, inserting $n$ as the last block will not change the statistic $\nsb$, but will increase the statistic $\rlb$ by 1. Obviously we have only one possible position for this insertion. The number of partitions of the set $[n-1]$ into $m-1$ blocks w.r.t. the statistics $\nsb$, $\nse$, $\rlb$, and $\rle$ is given by $\Sfun(n-1,m-1;\pv)$. Summarizing this case, we obtain 
$(\pa+(m-1)\pb)\Sfun(n-1,m-1;\pv)$.

Collecting two possibilities completes the proof.
\end{proof}

Comparing the recurrence relation of Proposition \ref{prop2-9} with the recurrence relation of Theorem \ref{thm4-2}, we can conclude that the function \eqref{Ffun} corresponds to the combinatorial model described in the Section \ref{Section-Comb}

\begin{theorem}\label{thm-expl}
\begin{align}
[\pa^{k-1-i}  
\pb^{i}\pc^{n-k-j}\pd^{j}]\Sfun(n,k;\pv) 
&
=\left|\left\{\pi\in\LLP_{n,k}\left|\begin{array}{l} \rlb(\pi)=k-1-i,\\ \nsb(\pi)=i,\\\rle(\pi)=n-k-j,\\ \nse(\pi)=j\end{array}\right.\right\}\right|\nn\\
&=\sone{n}{n-j}\stwo{n-j}{k}\sone{k}{k-i}.\nn
\end{align}
\end{theorem}
\begin{proof}
It follows from \eqref{Ffun}-\eqref{GF:UGF} and the Theorem~2.7 in \cite{Herscovici2019}.
\end{proof}

Now we can give the explicit form for the coefficients of our unified generating function \eqref{GF:UGF}
\begin{theorem} For $\pa,\pb,\pc,\pd$ such that 
$\pa\pb=0$ implies that either $\pa=0,\pb\neq0$ or $\pa\neq0,\pb=0$, and $\pc\pd=0$ implies that either $\pc=0,\pd\neq0$ or $\pc\neq0,\pd=0$, we have
\begin{align*}
\frac{1}{\pa}\pexp_\frac{\pb}{\pa}&\left[\pa x\left(\frac{1}{\pc}\pexp_\frac{\pd}{\pc}(\pc t)-\frac{1}{\pc}\right)\right]-\frac{1}{\pa}+\frac{\delta_{\pa,1}}{\pa}\\
&=\sum_{n=0}^\infty\sum_{k=1}^n\sum_{i=0}^{k-1}\sum_{j=0}^{n-k}\sone{n}{n-j}\stwo{n-j}{k}\sone{k}{k-i}\pa^{k-1-i}\pb^i\pc^{n-k-j}\pd^{j}x^k\frac{t^n}{n!}.
\end{align*}
\end{theorem}
In next Section we consider some applications of our results.
\section{Applications} \label{sec-app}
Let's start from considering some special cases of the parameters $\alpha$, $\beta$, $\lambda$, and $\mu$. Note that the cases 1--5 were considered in \cite{Herscovici2017, Herscovici2019}. All sequence numbers presented in the table are the numbers of corresponding sequences in On-Line Encyclopedia of Integer Sequences \cite{OEIS}.

{\renewcommand{\arraystretch}{2}%

\begin{longtable}{|c|c|c|c|c|c|p{0.3\textwidth}|}
\caption{Special values of parameters $\pa,\pb,\pc,\pd$ and generating functions of set partitions.}\label{longtable}\\
  \hline
   Case & $\alpha$  & $\beta$ & $\lambda$ & $\mu$ & $F(x,t;\pv)$ & Description\\
  \hline
  1 & 1 & $\beta$ & 1&  $\mu$ & $e_{\beta}[x(e_{\mu}(t)-1)]$ &  Enumeration of set partitions w.r.t. statistics $\nsb$ and $\nse$, see \cite{Herscovici2017, Herscovici2019}
 \\
 \hline 
   2 & 1 & 0 & 1&  0 & $e[x(e^t-1)]$ & Sets of sets \cite{Herscovici2017, Herscovici2019, Motzkin1971},
   sequence \href{https://oeis.org/A008277}{A008277}\\
  \hline
     3 & 1 & 0 & 1&  1 & $e\left(\frac{xt}{1-t}\right)$ & Lists of sets \cite{Herscovici2017, Herscovici2019, Motzkin1971}\\
  \hline
     4 & 1 & 1 & 1&  0 & $\frac{1}{1-x(e^t-1)}$ & Sets of lists \cite{Herscovici2017, Herscovici2019, Motzkin1971}, sequence \href{https://oeis.org/A019538}{A019538}\\
  \hline
     5 & 1 & 1 & 1&  1 & $\frac{1}{1-\frac{xt}{1-t}}$ & Lists of lists \cite{Herscovici2017, Herscovici2019, Motzkin1971}\\
  \hline
     6 & $\alpha$ & 1 & $\lambda$ &  1 & $-\log_{\alpha}[1+x\log_{\lambda}(1-t)]+\frac{\delta_{\alpha_1}}{\alpha}$ & Enumeration of set partitions with cyclically ordered blocks and cyclically ordered elements inside the blocks w.r.t. statistics $\rlb$ and $\rle$ \\
  \hline 
     7 & 0 & 1 & 0 &  1 & $-\log[1+x\log(1-t)]$ & Set partitions where both blocks and elements inside the block are cyclically ordered, sequence \href{https://oeis.org/A188881}{A188881}\\
  \hline  
     8 & 0 & 1 & 1 &  1 & $-\log[1-\frac{xt}{1-t}]$ & Set partitions with cyclically ordered blocks and ordered elements inside the blocks\\
  \hline
     9 & 1 & 1 & 0 &  1 & $\frac{1}{1+x\log(1-t)}$ & Set partitions with ordered blocks and cyclically ordered elements inside the blocks\\
  \hline
     10 & $\alpha$ & 1 & 1 &  $\mu$ & $-\log_{\alpha}[1-x(e_{\mu}(t)-1)]+\frac{\delta_{\alpha_1}}{\alpha}$ & Enumeration of set partitions with cyclically ordered blocks and ordered elements inside the blocks w.r.t. statistics $\rlb$ and $\nse$\\
  \hline
     11 & 0 & 1 & 1 &  0 & $-\log[1-x(e^t-1)]$ & Set partitions with cyclically ordered blocks and unordered elements inside the blocks\\
  \hline
     12 & 1 & $\beta$ & $\lambda$ &  1 & $e_{\beta}[-x\log_{\lambda}(1-t)]$ & Enumeration of set partitions with ordered blocks and cyclically ordered elements inside the blocks w.r.t. statistics $\nsb$ and $\rle$\\
  \hline
     13 & 1 & 0 & 0 &  1 & $(1-t)^{-x}$ & Set partitions with unordered blocks and cyclically ordered elements inside the blocks, sequence \href{https://oeis.org/A130534}{A130534}\\
  \hline
\end{longtable}}

\begin{example} Let us consider the case 6 from the Table \ref{longtable}. It corresponds to very certain values of the variables $\pb, \pd$, namely $\pb=\pd=1$, while there are no values assigned to the variables $\pa,\pc$. The corresponding to this case generating function is
\begin{align}
-\log_{\alpha}[1+x\log_{\lambda}(1-t)]+\frac{\delta_{\alpha_1}}{\alpha}=\sum_{n=0}^\infty\Uxp{n}{\pa,1,\pc,1}\tn{n}.
\end{align}
The first few polynomials $\Uxp{n}{\pa,1,\pc,1}$ are
\begin{align*}
\Uxp{1}{\pa,1,\pc,1}&=x,\\
\Uxp{2}{\pa,1,\pc,1}&=(1+\pc) x+(1+\pa) x^2
\\
\Uxp{3}{\pa,1,\pc,1}&=(2+3\pc+\pc^2)x+3(1+\pa+\pc+\pa\pc) x^2
+(2+3\pa+\pa^2) x^3\\
\Uxp{4}{\pa,1,\pc,1}&=(6+11\pc+6\pc^2+\pc^3)x+(1+\pa)(11+18\pc+7\pc^2)x^2\\
&+6(2+3\pa+\pa^2)(1+\pc)x^3+(6+11\pa+6\pa^2+\pa^3)x^4.
\end{align*}
Combinatorially, we enumerate set partitions where both blocks and elements inside the blocks are cyclically ordered w.r.t the statistics $\rlb$ and $\rle$. 
To classify a partition, we are allowed to start from any block and any element inside the blocks keeping the clockwise direction. Therefore the partition can be considered as ``in-line''  notation. Now we have $[\pa\pc^2 x^3]\Uxp{4}{\pa,1,\pc,1}=7$. It means there are exactly 6 partitions into 3 blocks with 1 right-to-left block and 2 right-to-left elements, namely
\begin{align*}
1\tilde{2}/\underline{3}\tilde{4},\quad 1\tilde{3}/\underline{2}\tilde{4},\quad 1\tilde{4}/2\tilde{3},\quad 1/\underline{2}\tilde{3}\tilde{4},\\
1\tilde{3}\tilde{4}/\underline{2},\quad1\tilde{2}\tilde{4}/\underline{3},\quad 1\tilde{2}\tilde{3}/\underline{4}.
\end{align*}
Here by $\tilde{2}$ we denote that $2$ is right-to-left element, and by $\underline{4}$ we denote that $4$ is right-to-left block. Applying the Theorem \ref{thm-expl} with $n=4, k=2, i=1, j=0$,  we obtain also
\begin{align}
[\pa^{1} %& 
\pc^{2}]\Sfun(n,k;\pv)%\nn\\
&=\sone{4}{4}\stwo{4}{2}\sone{2}{1}=7.\nn
\end{align}
\end{example}

\begin{remark}
The Table \ref{longtable} is written in terms of the set partitions. However, the case 13 is equivalent to the enumeration of permutations w.r.t. the number of cycles.
\end{remark}

\textbf{Acknowledgement}.
The author was supported by the NSF CAREER DMS-1753260. The author would like to thank Alan Sokal for interesting and encouraging conversations.

\end{document}